\documentclass[11pt]{amsart}
\usepackage{amssymb,cleveref,enumitem}
\theoremstyle{definition}
\newtheorem{tPp}{Proposition}[section]
\newtheorem{tDf}[tPp]{Definition}
\newtheorem{tEx}[tPp]{Example}
\newtheorem{tTh}[tPp]{Theorem}
\newtheorem{tRm}[tPp]{Remark}
\newtheorem{tLm}[tPp]{Lemma}
\newtheorem{tFc}[tPp]{Fact}
\newtheorem*{tMt}{Main Theorem}
\setlist[enumerate]{label*=(\arabic*)}
\renewcommand{\uppercasenonmath}[1]{}

\begin{document}
\title
{L-indistinguishability for covering groups of algebraic tori}
\author
{Yuki Nakata}
\address{Department of Mathematics, Kyoto University, Kyoto 606-8502,
Japan}
\email{nakata.24.nakata@outlook.jp}
\maketitle
\begin{abstract}
  A global packet may simultaneously contain an automorphic
  representation and a non-automorphic representation.
  The global $\mathcal S$-group is expected, and known in some cases,
  to specify the automorphic representations in each global packet.
  For a covering group of an algebraic torus, it is not obvious from
  the definition whether the analogue of a global $\mathcal S$-group
  has finite order.
  In this paper, we verify this finiteness
  for a Brylinski-Deligne covering group of a torus.
\end{abstract}
\tableofcontents

\section{Introduction}
This paper verifies the finiteness of the order of the global
$\mathcal S$-group for a 
certain covering group of
a torus over a number field $F$, which is not obvious from this
covering version of the definition.
Here, the base field $F$ is assumed to contain 
a full set $\mu_n$ of $n$-th roots of unity.

In addition to reductive algebraic groups, their covering groups
have also attracted arithmetical interest.
For example, Weil studied the Segal-Shale-Weil representations
of the 
double cover 
$\operatorname{Mp}_{2r}\to\operatorname{Sp}_{2r}$, so as to prove 
an identity now known as the Siegel-Weil 
formula~\cite{Weil_Sieg,KudRall_WeilSieg}.
Another notable instance is an Eisenstein series
on a double cover of $\operatorname{GL}_r$, which yields
a Rankin-Selberg integral representing the symmetric square L-function
of a cuspidal representation of 
$\operatorname{GL}_r$~\cite{BpG_SymS,Kab_TensXr,Take_TwSs}.
Such instances motivate us to pursue the Langlands program 
for covering groups.

In the familiar case of a reductive algebraic group $ G$ 
over a number field $F$,
we have the notion of the complex dual group $\widehat G$ of 
the group $ G$.
Then, a Langlands parameter $\phi$ valued in the L-group
$\vphantom{G}^{\mathrm L}G=\widehat G\rtimes\mathrm W_F$ defines
the centralizer 
$S_\phi=\operatorname{Cent}_{\widehat G}(\operatorname{Im}\phi)$
and the finite quotient
$\mathcal S_\phi=\left.S_\phi\middle/
S_\phi^0\operatorname Z(\widehat G)^{\Gamma_F}\right.$.
Here, $S_\phi^0$ is the connected component of the identity
in the complex reductive group $S_\phi$, and
$\operatorname Z(\widehat G)^{\Gamma_F}$
is the subgroup consisting of fixed points in the center 
$\operatorname Z(\widehat G)\subset\widehat G$
under the action of the absolute Galois group $\Gamma_F$. 
In this paper, we investigate an analogue of the group 
$\mathcal S_\phi$ for a covering group of an adele group.

For a covering group $\widetilde{T(F)}\to T(F)$ of a torus over a
local field $F$, Weissman formulated an analogue of the local
$\mathcal S$-group, which he called the packet
group~\cite{Weiss_loc_tori}.
In his formulation, the cover $\widetilde{T(F)}$ is assumed to be
a functorial one defined by Brylinski and Deligne~\cite{Bry_Del}.
The local $\mathcal S$-group for the covering group is explicitly
determined under mild assumptions~\cite{Ny_LcS,Weiss_loc_tori}.

In this paper, we formulate the global $\mathcal S$-group for a
Brylinski-Deligne covering group $\widetilde{T(\mathbb A)}\to
T(\mathbb A)$ of a torus over a number field $F$.
From this covering version of the definition, it is not obvious
whether the global $\mathcal S$-group for $\widetilde{T(\mathbb A)}$
has finite order.
Our main theorem verifies the finiteness of this order:

\begin{tMt}[\Cref{tMt}]
  Let $F$ be an algebraic number field
  containing a full set $\mu_n$ of $n$-th roots of unity,
  $\mathbb A
  $ the adele ring,
  and $T$ an algebraic torus defined over $F$.
  Let $\widetilde{T(\mathbb A)}\to T(\mathbb A)$ be a 
  Brylinski-Deligne covering group of the adelic torus.
  Then, the global $\mathcal S$-group 
  for the covering group $\widetilde{T(\mathbb A)}$ has finite order.
\end{tMt}

In the next section, we review the definition and 
the classification theorem for Brylinski-Deligne covering groups
of an algebraic torus $ T$. 
Before global discussion, 
we also review the formulation of a local packet and the local
$\mathcal S$-group for 
the covering group in \Cref{sLoc}.
Then we formulate a global packet and the global $\mathcal S$-group in
\Cref{sGlb}.
In the last section, we prove the main theorem by applying a local
statement.

\subsection*{Notation}
Let $\mathbb Z$ denote the ring of rational integers,
and $\mathbb C$ the field of complex numbers.
For a map $f\colon A\to B$ and a subset $A'\subset A$,
let $f|_{A'}$ denote the restriction, as usual.
We sometimes use this symbol to denote the pullback to $A'$, even if
the map $A'\to A$ is not an inclusion.
For a group $G$, let $\operatorname Z(G)\subset G$ denote the center.
For a field $F$, let $F^\times$ denote the multiplicative group
consisting of non-zero elements.

Throughout this paper, $F$ denotes a field.
If $F$ is a local or global field, then
we always assume that 
$F$ contains a primitive $n$-th root of unity, and 
let $\mu_n\subset F^\times$ denote the 
subgroup consisting of $n$-th roots of unity.
From \Cref{sGlb}, $F$ is a number field.
Then, at each place $v$ of $F$, let $F_v$ denote the completion, and
$\mathcal O_v\subset F_v$ for its valuation ring.
Let $\mathbb A$ denote the adele ring of 
$F$.
For a set $\Sigma$ of places, 
let $\mathbb A_\Sigma$ denote the $\Sigma$-adele ring, which fits in
a decomposition 
$\mathbb A=\mathbb A_\Sigma\times\prod_{v\in\Sigma}F_v$.
If $\Sigma$ contains all infinite places of $F$, then $\mathcal
O_\Sigma=\bigcap_{v\notin\Sigma}(F\cap\mathcal O_v)$ denotes the
ring of $\Sigma$-integers.

\section{Brylinski-Deligne covering group}
In the literature on the Langlands program for
general covering groups~\cite{GG_Weiss,Weiss_L}, 
one often focuses on certain covering groups defined by
Brylinski and Deligne~\cite{Bry_Del}.
More recently, Zhao refined them to define 
``\'etale metaplectic covers,'' 
but we do not discuss it.
Here, we just review the definition and 
the classification theorem for Brylinski-Deligne covering groups
of an algebraic torus $ T$ defined over a 
field $F$.

A Brylinski-Deligne covering group is,
roughly,
a central extension by Milnor's group $\operatorname K_2$
in algebraic K-theory, or its variation.
The groups $\operatorname{K}_2$ are 
generalized by Quillen~\cite{Qill_hi}
to be defined over schemes, and form a presheaf on the big Zariski
site of 
a base field $F$.
Let $\mathbb K_2$ be the 
sheaf associated with the presheaf $\operatorname K_2$.

\begin{tDf}[Brylinski-Deligne~\cite{Bry_Del}]
  Let $T$ 
  be an algebraic torus over a field $F$.
  Then, a Brylinski-Deligne $\mathbb K_2$-extension of 
  $T$ is a central extension
  \[1\to\mathbb{K}_2\to\widetilde{{T}}\to{T}\to 1\]
  as a sheaf of groups on the big Zariski site of 
  $\operatorname{Spec}F$.
\end{tDf}

Here, we allow this extension to be non-abelian.
Various covering groups induced from it in the following way
are called Brylinski-Deligne
covering groups.

\begin{tEx}
  \begin{enumerate}
    \item\label{e_BD_sc}
      Let $F$ be any field.
      Then, the sections
      \[
        1\to\operatorname K_2(F)\to\widetilde{{T}}(F)\to{T}
        \to 1
      \]
      of a Brylinski-Deligne $\mathbb K_2$-extension
      $\widetilde{T}\to T$
      over $\operatorname{Spec}F$ form a central extension as 
      an abstract group.
    \item\label{e_BDlc}
      Let $F$ be a local field
      containing a full set $\mu_n$ of $n$-th roots of unity.
      Then the push-out 
      \[1\to\mu_n\to\widetilde{T(F)}\to{T(F)}\to 1\]
      of the extension~\ref{e_BD_sc}
      via the Hilbert symbol $\operatorname K_2(F)\to\mu_n$ is called
      an $n$-fold \emph{Brylinski-Deligne covering group} of $T(F)$.
      Here, if $F$ is the field $\mathbb C$ of complex numbers,
      then we regard the Hilbert symbol as the trivial map.

      This covering group is indeed a topological 
      one~\cite[Construction 10.3]{Bry_Del}.
    \item
      Let $F$ be a number field
      containing a full set $\mu_n$ of $n$-th roots of unity, and
      $\mathbb A=\mathbb A_F$ the adele ring. 
      At each place $v$ of 
      $F$,~\ref{e_BDlc} gives a local 
      covering group 
      $ 
        \mu_n\to\widetilde{T(F_v)}\to{T}(F_v)
      $.

      At almost every place $v$ of 
      $F$, Brylinski and Deligne found a split monomorphism 
      $ T(\mathcal O_v)\to\widetilde{T(F_v)}$ in this cover.
      Thus, we may construct the restricted direct product 
      $\prod_v'\widetilde{T(F_v)}$ with respect to the family 
      $\{ T(\mathcal O_v)\}_v$ of compact groups.
      This restricted direct product is a central extension
      $1\to\bigoplus_v\mu_n\to\prod_v'\widetilde{T(F_v)}\to
      {T}(\mathbb A)\to 1$, and gives a covering group
      \[
        1\to\mu_n\to\widetilde{ T(\mathbb A)}\to
         T(\mathbb A)\to1
      \]
      by taking the push-out via the 
      product map $\bigoplus_v\mu_n\to\mu_n$.
      This is called an $n$-fold
      \emph{Brylinski-Deligne covering group of the adele group} 
      $ T(\mathbb A)$. 

      A Brylinski-Deligne covering group $\widetilde{T(\mathbb A)}\to
      T(\mathbb A)$ is known to split canonically over the diagonal
      embedding of $T(F)$.
      By this splitting, we regard $T(F)$ as a subgroup of the cover
      $\widetilde{T(\mathbb A)}$.
  \end{enumerate}
\end{tEx}

Brylinski and Deligne generally defined such covering groups
for any algebraic group;
a metaplectic group $\mathrm{Mp}_{2r}\to\mathrm{Sp}_{2r}$ is one of
the well-known examples.

For any torus $ T$ over a field $F$,
we have a finite Galois extension $L/F$
such that $ T$ splits over $L$.
Then the following classification theorem gives a bilinear form
$Y\times Y\to\mathbb Z$ on the cocharacter lattice,
which is an essential ingredient to 
define a packet of representations of a Brylinski-Deligne covering
group.

\begin{tTh}[Classification Theorem of 
  $\mathbb K_2$-extensions~\cite{Bry_Del}]
  \label{etBil}
  The following two Picard categories are equivalent:
  \begin{enumerate}
    \item the category of Brylinski-Deligne $\mathbb K_2$-extensions
      of the torus $ T$;

    \item the category of the following data $(Q,\mathcal{E})$:
      \begin{itemize}
        \item 
          $Q\colon Y\to\mathbb{Z}$ is a Galois-invariant quadratic 
          form on the cocharacter lattice, 
          which gives a bilinear form 
          $B\colon Y\times Y\to\mathbb{Z}$ by 
          \[(y,y')\mapsto Q(y+y')-Q(y)-Q(y').\]
        \item 
          $1\to L^\times\to\mathcal{E}\stackrel{\overline{(\ )}}{\to}Y
          \to 1$ is a Galois-equivariant central extension such that
          for any $\epsilon,\zeta\in\mathcal{E}$,
          \[
            \epsilon\zeta\epsilon^{-1}\zeta^{-1}=
            (-1)^{B(\overline{\epsilon},\overline{\zeta})}
          \]
          in $L^\times$.
          Here, the symbols $\overline\epsilon$ and $\overline\zeta$
          respectively denote the images of $\epsilon$ and $\zeta$
          under the surjection $\mathcal E\to Y$.
        \item
          For any data $(Q,\mathcal E)$ and $(Q',\mathcal E')$,
          \[
            \operatorname{Hom}((Q,\mathcal{E}),(Q',\mathcal{E}'))=
            \begin{cases}
              \emptyset &\text{if }Q\neq Q'\\
              \{\text{morphisms }\mathcal{E}\to\mathcal{E}'
              \text{ of extensions}\}
              &\text{if }Q=Q'.
            \end{cases}
          \]
          Here, morphisms $\mathcal{E}\to\mathcal{E}'$ are also
          required to be Galois-equivariant.
      \end{itemize}
  \end{enumerate}
\end{tTh}

More generally,
Brylinski and Deligne proved this type of classification
theorem for $\mathbb K_2$-extensions of reductive algebraic 
groups~\cite{Bry_Del}.

\section{Local packet and $\mathcal S$-group}
\label{sLoc}
Before global discussion, 
we review Weissman's formulation~\cite{Weiss_L} of a local
packet and the $\mathcal S$-group for an $n$-fold
Brylinski-Deligne covering group $
\widetilde{T(F)}\to T(F)$ of a torus over a local field $F$ 
containing a full set $\mu_n$ of $n$-th roots of unity.
To define a local packet 
for the covering group $\widetilde{T(F)}
$,
Weissman introduced the following sublattice $Y^\#$ of the cocharacter
lattice $Y$ of the torus $T$.

In this section, let $F$ be a local field  
containing a full set $\mu_n$ of $n$-th roots of unity.

\begin{tDf}[Weissman~\cite{Weiss_loc_tori}]
  Let $\mu_n\to\widetilde{T(F)}\to T(F)$ be a
  Brylinski-Deligne covering group of a torus $T$ over $F$.
  As stated in \Cref{etBil}, the original $\mathbb K_2$-extension
  $\widetilde T\to T$ defining $\widetilde{T(F)}$ yields a quadratic
  form $Q\colon Y\to\mathbb Z$ as a first invariant, which defines a
  bilinear form $B\colon Y\times Y\to\mathbb Z$.
  Then we define a sublattice $Y^\#\subset Y$ by
  \[Y^\#=\{y\in Y\mid B(y,Y)\subset n\mathbb Z\}.\]
\end{tDf}

The inclusion $Y^\#\hookrightarrow Y$ induces an isogeny
$T^\#\stackrel\iota\to T$ of tori.
Let $\widetilde{T^\#(F)}\to T^\#(F)$ be the pullback of the $n$-fold
Brylinski-Deligne covering group $\widetilde{T(F)}\to T(F)$ via the
isogeny $\iota$.
By functoriality, 
$\widetilde{T^\#(F)}$ is also an $n$-fold Brylinski-Deligne cover, 
which is defined by the pullback $\mathbb K_2$-extension
$\widetilde{T^\#}\to T^\#$ from $\widetilde T\to T$.

Throughout this paper, a representation of a covering group is assumed
to be genuine.
To be precise, we fix an embedding 
\[j\colon\mu_n\to\mathbb{C}^\times\]
of the cyclic group of order $n$.
Recall that
a representation of a covering group $\mu_n\to\widetilde G\to G$ 
of a topological group is \emph{genuine} 
if $\mu_n\subset\widetilde{G}$ acts as 
multiplication by scalars in $\mathbb{C}$ via the embedding $j$.
That is, for any $\zeta\in\mu_n$ and any vector 
$v$ in the representation space, one has
\[\zeta\cdot v=j(\zeta)v.\]
A \emph{genuine character} 
$\chi\colon\widetilde G\to\mathbb C^\times$ is
a character that is genuine as a representation,
which means $\chi|_{\mu_n}=j$.

\begin{tDf}
  \label{tLids}
  Let $\widetilde{T(F)}\to T(F)$ be an $n$-fold
  Brylinski-Deligne covering group of a torus. 
  \begin{enumerate}
    \item
      Two irreducible representations $\pi_1$ and $\pi_2$ of 
      the covering group $\widetilde{T(F)}$ are said to be 
      \emph{L-indistinguishable} if their pullbacks to 
      $\widetilde{T^\#(F)}$ coincide:
      \[ 
      \pi_1|_{\widetilde{T^\#(F)}}=\pi_2|_{\widetilde{T^\#(F)}}.  
      \]
      By abuse of notation, we use vertical bars to indicate pullbacks,
      like restrictions.
    \item
      This L-indistinguishability of course defines an equivalence
      relation, and each equivalence class is called a \emph{local
      packet}, or \emph{pouch}~\cite{Weiss_L}.
  \end{enumerate}
\end{tDf}

\begin{tRm}
  In \Cref{tLids}, the pullbacks $\pi_i|_{\widetilde{T^\#(F)}}$
  acts as a character on the representation space, for $i=1,2$.
  Indeed, the image $\iota(\widetilde{T^\#(F)})$ in the 
  cover $\widetilde{T(F)}$ is known to be central~\cite[Theorem 1.3]
  {Weiss_LG_tori}, i.e., included in the center $\operatorname
  Z(\widetilde{T(F)})$.
  Thus $\pi_i|_{\widetilde{T^\#(F)}}$ acts as the pullback of the
  central character of $\pi_i$.
\end{tRm}
Note that the image $\iota({T^\#(F)})$ has finite index in ${T(F)}$.
Hence the image $\iota(\widetilde{T^\#(F)})$ of the cover also has
finite index in $\widetilde{T(F)}$, and especially in the center 
$\operatorname Z(\widetilde{T(F)})$.

\begin{tDf}
  Let $\widetilde{T(F)}\to T(F)$ be an $n$-fold
  Brylinski-Deligne covering group. 
  Then we define the finite abelian group 
  $
  \mathcal S_{\widetilde{T(F)}}$ 
  as the quotient
  \[\mathcal S_{\widetilde{T(F)}}=
  \left.\operatorname Z(\widetilde{T(F)})\middle/
  \iota(\widetilde{T^\#(F)})\right.
  .\]
  Weissman called this the \emph{packet group}~\cite{Weiss_loc_tori}.
\end{tDf}

\begin{tPp}
  Let $\Pi$ be a local packet of representations of the covering group
  $\widetilde{T(F)} $.
  Then, fixing a representation $\pi_0$ in a local packet $\Pi$ yields
  a 
  bijection 
  \[
    \Pi\stackrel{1:1}\to
    \operatorname{Hom}(\mathcal S_{\widetilde{T(F)}},\mathbb C^\times)
    ,
  \]
  which we write $\pi\mapsto\langle\quad,\pi\rangle$.
\end{tPp}

\begin{proof}
  Let $\omega(\Pi)$ be the set of central characters 
  $\operatorname Z(\widetilde{T(F)})\to\mathbb C^\times$ 
  of representations belonging to $\Pi$.
  Recall that any representation of the covering group 
  $\widetilde{T(F)}$ is assumed to be genuine.

  By an analogy~\cite{Weiss_loc_tori} of the Stone-von Neumann theorem
  for 
  $\widetilde{T(F)}$, the surjection $\Pi\to\omega(\Pi)$ is
  indeed bijective.
  Further, if $\chi_0$ denotes the central character of $\pi_0$, then
  $\omega(\Pi)$ coincides with the set of genuine characters $\chi$ of
  the center $\operatorname Z(\widetilde{T(F)})$ such that the
  restrictions $\chi|_{\iota(\widetilde{T^\#(F)})}$ are equal to the
  restriction $\chi_0|_{\iota(\widetilde{T^\#(F)})}$.
  Then the map 
  $\omega(\Pi)\to
  \operatorname{Hom}(\mathcal S_{\widetilde{T(F)}},\mathbb C^\times)$
  defined by $\chi\mapsto\chi\chi_0^{-1}$ is bijective, so is the
  composite
  $\Pi\to
  \operatorname{Hom}(\mathcal S_{\widetilde{T(F)}},\mathbb C^\times)$.
\end{proof}

\section{Global packet and $\mathcal S$-group}
\label{sGlb}
In a familiar way,
we define a global packet for a Brylinski-Deligne covering group
$\widetilde{T(\mathbb A)}\to T(\mathbb A)$ of a torus over
a number field $F$ 
containing a full set $\mu_n$ of $n$-th roots of unity.
Here, $\mathbb A=\mathbb A_F$ denotes the adele ring of $F$.
Then we formulate the global $\mathcal S$-group for the covering group
$\widetilde{T(\mathbb A)}$.

From this section, $F$ denotes a number field  
containing a full set $\mu_n$ of $n$-th roots of unity.

\begin{tDf}
  Let $\widetilde{T(\mathbb A)}\to T(\mathbb A)$ be an $n$-fold
  Brylinski-Deligne covering group of a torus.
  Then, we define two irreducible unitary representations
  $\pi=\bigotimes_v'\pi_v$ and $\pi'=\bigotimes_v'\pi'_v$ of 
  $\widetilde{T(\mathbb A)}$ to be \emph{L-indistinguishable} 
  if $\pi_v$ and $\pi'_v$ are L-indistinguishable at any place $v$
  and equivalent at almost every place $v$.
  A \emph{global packet}, or \emph{pouch}, is an equivalence class
  for this relation.
\end{tDf}

To define the global $\mathcal S$-group, we examine the center
$\operatorname Z(\widetilde{T(\mathbb A)})$ of a Brylinski-Deligne
covering group $\widetilde{T(\mathbb A)}\to T(\mathbb A)$.
By definition, the image $\operatorname Z^\dagger(\mathbb A)\subset
T(\mathbb A)$ of the center $\operatorname Z(\widetilde{T(\mathbb A)})
$ is the restricted direct product of the images $\operatorname
Z^\dagger(F_v)\subset T(F_v)$ of the centers $\operatorname
Z(\widetilde{T(F_v)})$, with respect to the family $\{\operatorname
Z^\dagger(F_v)\cap T(\mathcal O_v)\}_v$ of intersections.
Then, the local inclusions 
$\operatorname Z^\dagger(F_v) \supset\iota(T^\#(F_v))$ 
imply the global inclusion
$\operatorname Z^\dagger(\mathbb A) \supset\iota(T^\#(\mathbb A))$,
which verifies the inclusion
$\operatorname Z(\widetilde{T(\mathbb A)})
\supset\iota(\widetilde{T^\#(\mathbb A)})$ between covers.

Then we may define the global $\mathcal S$-group in a way similar to
the local case.
\begin{tDf}
  Let $\widetilde{T(\mathbb A)}\to T(\mathbb A)$ be a 
  Brylinski-Deligne covering group of a torus.
  Then, we define the abelian group 
  $\mathcal S_{\widetilde{T(\mathbb A)}}$ as the quotient
  \begin{align*}
    \mathcal S_{\widetilde{T(\mathbb A)}}
    &=\left.\operatorname Z(\widetilde{T(\mathbb A)})\cap T(F)
    \middle/ \iota(\widetilde{T^\#(\mathbb A)})\cap T(F)\right.\\
    &\cong\left.\operatorname Z^\dagger(\mathbb A)\cap T(F)
    \middle/ \iota({T^\#(\mathbb A)})\cap T(F)\right..
  \end{align*}
  Here, we regard $T(F)$ as a subgroup of the cover
  $\widetilde{T(\mathbb A)}$ by the canonical splitting over the
  diagonal embedding $T(F)\hookrightarrow T(\mathbb A)$.
\end{tDf}

At each place $v$ of $F$, the group 
$\operatorname Z(\widetilde{T(\mathbb A)})\cap T(F)$ naturally embeds
in $\operatorname Z(\widetilde{T(F_v)})$, and 
$\iota(\widetilde{T^\#(\mathbb A)})\cap T(F)$ in
$\iota(\widetilde{T^\#(F_v)})$.
These embeddings define a map 
$\mathcal S_{\widetilde{T(\mathbb A)}}\to
\mathcal S_{\widetilde{T(F_v)}}$, denoted by $s\mapsto s_v$.

We suppose that a global packet $\Pi$ contains an automorphic 
representation $\pi^0$ of $\widetilde{T(\mathbb A)}$.
Here, other representation in $\Pi$ may not be automorphic.
We describe the condition for automorphy 
in terms of the global $\mathcal S$-group as follows.

To define the local characters 
$\langle\quad,\pi_v\rangle\colon\mathcal S_{\widetilde{T(F_v)}}\to
\mathbb C^\times$, we chose one of the representations 
in the local packet. 
Instead, we now fix one automorphic representation 
$\pi^0=\bigotimes'_v\pi^0_v$ in the global packet $\Pi$.
At each place $v$, we use the local component $\pi^0_v$ to define
the characters $\langle\quad,\pi_v\rangle$.

\begin{tLm}
  Let $\Pi$ be a global packet with a fixed automorphic representation
  $\pi^0\in\Pi$.
  Take any $s\in\mathcal S_{\widetilde{T(\mathbb A)}}$ and
  $\pi\in\Pi$.
  Then we have $\langle s_v,\pi_v\rangle=1$ for almost every place
  $v$.
\end{tLm}
\begin{proof}
  Let $\dot s\in\operatorname Z(\widetilde{T(A)})\cap T(F)$ be a
  representative of $s\in\mathcal S_{\widetilde{T(\mathbb A)}}$.
  Then for almost every place $v$ of $F$, the $v$-component $\dot s_v$
  of $\dot s$ belongs to $\operatorname Z(\widetilde{T(F_v)})\cap
  T(\mathcal O_v)$.
  Recall that for almost every $v$, the cover $\widetilde{T(F_v)}$
  canonically splits over $T(\mathcal O_v)$, which allow us to regard
  $T(\mathcal O_v)$ as a subgroup of $\widetilde{T(F_v)}$.

  Let $\chi$ be the central character $\operatorname
  Z(\widetilde{T(\mathbb A)})\to\mathbb C^\times$ of the
  representation $\pi$, and $\chi^0$ that of $\pi^0$.
  Then for almost every place $v$, the $v$-components $\chi_v$ and
  $\chi^0_v$ are trivial on $\operatorname Z(\widetilde{T(F_v)})\cap
  T(\mathcal O_v)$.
  Note that $\chi_v$ is the central character of the $v$-component
  $\pi_v$, and similarly for the character $\chi^0_v$.
  Therefore, for almost every place $v$, the value 
  $\langle s_v,\pi_v\rangle=\chi_v(\chi^0_v)^{-1}(s_v)$ is trivial.
\end{proof}

Hence we may define 
$\langle s,\pi\rangle=\prod_v\langle s_v,\pi_v\rangle$ for 
$s\in\mathcal S_{\widetilde{T(\mathbb A)}}$ and $\pi\in\Pi$.
\begin{tPp}
  Let $\Pi$ be a global packet with a fixed automorphic representation
  $\pi^0\in\Pi$.
  \begin{enumerate}
    \item 
      Let $s\in\mathcal S_{\widetilde{T(\mathbb A)}}$ and $\pi\in\Pi$.
      Then, the value $\langle s,\pi\rangle$ is independent of 
      the choice of a representation $\pi^0$.
    \item The pairing $\langle\quad{,}\quad\rangle$ induces 
      a surjection of $\Pi$ onto the dual group of 
      $\mathcal S_{\widetilde{T(\mathbb A)}}$.
  \end{enumerate}
\end{tPp}

\begin{proof}
  Let $\omega(\Pi)$ be the set of central characters 
  $\operatorname Z(\widetilde{T(\mathbb A)})\to\mathbb C^\times$ 
  of representations belonging to $\Pi$.
  By an analogy~\cite[Theorem 4.5]{Weiss_LG_tori} of the Stone-von
  Neumann theorem for the covering group $\widetilde{T(\mathbb A)}$,
  the surjection \[\Pi\to\omega(\Pi)\] is indeed bijective.
  Further, if $\chi^0$ denotes the central character of $\pi^0$, then
  $\omega(\Pi)$ coincides with the set of genuine characters $\chi$ of
  the center $\operatorname Z(\widetilde{T(\mathbb A)})$ such that the
  restrictions $\chi|_{\iota(\widetilde{T^\#(\mathbb A)})}$ are equal
  to the restriction $\chi^0|_{\iota(\widetilde{T^\#(\mathbb A)})}$.

  Then the map 
  \[
    \omega(\Pi)\to
    \operatorname{Hom}\left(\operatorname Z(\widetilde{T(\mathbb A)})
    /\iota(\widetilde{T^\#(\mathbb A)}),\mathbb C^\times\right)
  \]
  defined by $\chi\mapsto\chi(\chi^0)^{-1}$ is bijective.
  Note that the canonical map 
  $\mathcal S_{\widetilde{T(\mathbb A)}}
  \to\operatorname Z(\widetilde{T(\mathbb A)})
  /\iota(\widetilde{T^\#(\mathbb A)})$
  is injective.
  Then the composition
  \[
    \omega(\Pi)\to
    \operatorname{Hom}\left(\operatorname Z(\widetilde{T(\mathbb A)})
    /\iota(\widetilde{T^\#(\mathbb A)}),\mathbb C^\times\right)
    \to\operatorname{Hom}(\mathcal S_{\widetilde{T(\mathbb A)}},
    \mathbb C^\times)
  \]
  is a surjection, mapping $\chi$ to 
  $\chi(\chi^0)^{-1}
  |_{\operatorname Z(\widetilde{T(\mathbb A)})\cap T(F)}
  =\chi|_{\operatorname Z(\widetilde{T(\mathbb A)})\cap T(F)}$

  Therefore the composite 
  $\Pi\to\omega(\Pi)\to
  \operatorname{Hom}(\mathcal S_{\widetilde{T(\mathbb A)}},
  \mathbb C^\times)$
  is surjective and independent of the choice of a representation
  $\pi^0$.
\end{proof}

Moreover:
\begin{tPp}
  Let $\Pi$ be a global packet containing at least one automorphic
  representation.
  Then, a unitary representation $\pi\in\Pi$ is automorphic
  if and only if $\langle\quad,\pi\rangle$ is the trivial character
  of $\mathcal S_{\widetilde{T(\mathbb A)}}$.
\end{tPp}
\begin{proof}
  If $\pi$ is an automorphic representation, then the central
  character $\chi\colon\operatorname Z(\widetilde{T(\mathbb A)})
  \to\mathbb C^\times$ is trivial on $\operatorname
  Z(\widetilde{T(\mathbb A)})\cap T(F)$.
  Conversely, if the central character $\chi$ of a unitary
  representation $\pi\in\Pi$ is trivial on
  $\operatorname Z(\widetilde{T(\mathbb A)})\cap T(F)$, then $\pi$ is
  automorphic, since the space of automorphic forms on
  $\widetilde{T(\mathbb A)}$ with central character $\chi$ decomposes
  into a sum of the copies of $\pi$~\cite{Weiss_LG_tori}.
  Finally, note that the triviality of the central character $\chi$ on
  $\operatorname Z(\widetilde{T(\mathbb A)})\cap T(F)$ just means the
  triviality on $\mathcal S_{\widetilde{T(\mathbb A)}}$.
\end{proof}

\section{Proof of Main Theorem}
Our main theorem claims that the abelian group
\[
  \mathcal S_{\widetilde{T(\mathbb A)}}=
  \left.(\operatorname{Z}^\dagger(\mathbb A)\cap T(F))
  \middle/(\iota( T^\#(\mathbb A))\cap T(F))\right.
\]
has finite order.
Recall that this group determines the automorphic representations in a
global packet for an $n$-fold Brylinski-Deligne covering group
$\widetilde{T(\mathbb A)}\to T(\mathbb A)$ of a torus over a number
field $F$ containing a full set $\mu_n$ of $n$-th roots of unity.
We prove the main theorem by reducing it to the following finiteness
statement.

\begin{tLm}
  \label{tFg}
  Suppose that a finite set $\Sigma$ of places of $F$ contains all
  infinite places.
  Then, the subgroup $\iota(T^\#(\mathcal O_\Sigma))$ of $T(\mathcal
  O_\Sigma)$ has finite index.
\end{tLm}
\begin{proof}
  The 
  group $T(\mathcal O_\Sigma)$ of $\Sigma$-integer points is known to
  be finitely generated~\cite[Theorem 5.37]{PltRR}.
  Hence, the subgroup 
  $(T(\mathcal O_\Sigma))^n\subset T(\mathcal O_\Sigma)$ has finite
  index.
  The definition of $Y^\#$ implies the inclusions $nY\subset
  Y^\#\subset Y$.
  They induce the maps 
  \[
    T(\mathcal O_\Sigma)\stackrel n\to
    T^\#(\mathcal O_\Sigma)\stackrel\iota\to T(\mathcal O_\Sigma),
  \]
  which ensures the inclusion $(T(\mathcal O_\Sigma))
  ^n\subset\iota(T^\#(\mathcal O_\Sigma))$.
  Therefore, 
  $\iota(T^\#(\mathcal O_\Sigma))$ also has finite index in
  $T(\mathcal O_\Sigma)$.
\end{proof}

To reduce the matter to the integer points, we focus on a certain
relation regarding the local $\mathcal S$-groups, which is indeed the
essential part of the proof of our main theorem.
Recall that at any place $v$ of the number field $F$, the symbol
$
\operatorname Z^\dagger(F_v)\subset T(F_v)$ denotes the image of the
center $\operatorname Z(\widetilde{T(F_v)})$ of the cover
$\widetilde{T(F_v)}
$.

\begin{tPp}
  \label{tModO}
  We fix a finite Galois extension $L/F$ where the torus $T$ splits
  over $L$.
  Let $v\nmid n$ be a finite place of $F$ such that 
  the ramification index $e$ of $L/F$ at $v$ is relatively prime to 
  the degree $n$ of the cover.
  Then we have an inclusion
  \[
    \operatorname Z^\dagger(F_v)\subset
    \iota(T^\#(F_v))\cdot T(\mathcal O_v).
  \]
\end{tPp}
\begin{proof}
  This was proved in the previous work~\cite{Ny_LcS} of the author.
\end{proof}

In the proof of our main theorem, we apply the following.
\begin{tFc}
  \label{tCln}
  Let $T$ be a torus over a number field $F$.
  If $\Sigma$ is a sufficiently large finite set of places of $F$,
  then 
  \[
    T(\mathbb A_\Sigma)=T(F)\cdot\prod_{v\notin \Sigma}T(\mathcal O_v)
    .
  \]
\end{tFc}

\begin{proof}
  The class number of a torus $T$ is known to be finite~\cite[Theorem
  5.8]{PltRR}.
  Thus, we have a finite subset $C
  \subset T(\mathbb A)$ such that 
  \[
    T(\mathbb A)=T(F)\cdot C\cdot\prod_{v\colon\text{infinite}}T(F_v)
    \cdot\prod_{v\colon\text{finite}}T(\mathcal O_v).
  \]
  Since $C$ is finite, we find a finite set $\Sigma$ of places such
  that 
  $C\subset\prod_{v\in\Sigma}T(F_v)
  \cdot\prod_{v\notin\Sigma}T(\mathcal O_v)$.
  Hence,
  \[
    T(\mathbb A)=T(F)\cdot\prod_{v\in\Sigma}T(F_v)
    \cdot\prod_{v\notin\Sigma}T(\mathcal O_v),
  \]
  whose $\mathbb A_\Sigma$-part is the desired equality.
\end{proof}

Now we can prove the main theorem:
\begin{tTh}
  \label{tMt}
  Let $F$ be an algebraic number field
  containing a full set $\mu_n$ of $n$-th roots of unity,
  $\mathbb A
  $ the adele ring,
  and $T$ an algebraic torus defined over $F$.
  Let $\widetilde{T(\mathbb A)}\to T(\mathbb A)$ be a 
  Brylinski-Deligne covering group of the adelic torus.
  Then, the group $\mathcal S_{\widetilde{T(\mathbb A)}}$ has finite
  order.
\end{tTh}

\begin{proof}
  Let $\Sigma$ be a sufficiently large finite set of places of $F$,
  so that 
  we may apply \Cref{tCln} and \Cref{tModO} at any place 
  $v\notin\Sigma$.
  Then \Cref{tModO} implies the inclusion
  \[
    \operatorname Z^\dagger(\mathbb A_\Sigma)\subset
    \iota(T^\#(\mathbb A_\Sigma))\cdot 
    \prod_{v\notin\Sigma}T(\mathcal O_v),
  \]
  and \Cref{tCln} implies the equality
  $T^\#(\mathbb A_\Sigma)=
  T^\#(F)\cdot\prod_{v\notin \Sigma}T^\#(\mathcal O_v)$.
  Thus,
  $\operatorname Z^\dagger(\mathbb A_{\Sigma})\subset
  \iota(T^\#(F))\cdot
  \prod_{v\notin \Sigma}T(\mathcal O_v)$, and therefore
  \begin{align*}
    T(F)\cap\operatorname Z^\dagger(\mathbb A)\subset
    T(F)\cap\operatorname Z^\dagger(\mathbb A_\Sigma)
    &\subset\iota(T^\#(F))\cdot T(\mathcal O_\Sigma)\\
    &\subset
    \left(\iota( T^\#(\mathbb A))\cap T(F)\right)\cdot
    T(\mathcal O_\Sigma).
  \end{align*}

  In the quotients
  \[
    ( T(F)\cap\operatorname Z^\dagger(\mathbb A))
    /\iota( T^\#(\mathcal O_\Sigma))
    \subset
    (\iota( T^\#(\mathbb A))\cap T(F)
    )/\iota( T^\#(\mathcal O_\Sigma))
    \ \cdot\ 
    T(\mathcal O_\Sigma)
    /\iota( T^\#(\mathcal O_\Sigma)),
  \]
  the last factor 
  $ T(\mathcal O_\Sigma)/\iota( T^\#(\mathcal O_\Sigma))$
  is finite,
  by \Cref{tFg}.
  Therefore, the 
  group 
  \[
    \mathcal S_{\widetilde{T(\mathbb A)}}=
    \left.(\operatorname{Z}^\dagger(\mathbb A)\cap T(F))
    \middle/(\iota( T^\#(\mathbb A))\cap T(F))\right.
  \]
  also has finite order.
\end{proof}

\section*{Acknowledgements}
The author would like to express gratitude to 
Professor Tamotsu Ikeda for the insightful discussions. 
This work was supported by JSPS KAKENHI Grant Number JP23KJ1298,
and also by the Research Institute for Mathematical
Sciences, an International Joint Usage/Research Center located in
Kyoto University.


\begin{thebibliography}{10}

\bibitem{Bry_Del}
J.-L. Brylinski and P.~Deligne.
\newblock Central extensions of reductive groups by {$\bold K_2$}.
\newblock {\em Publ. Math. Inst. Hautes \'Etudes Sci.}, (94):5--85, 2001.

\bibitem{BpG_SymS}
D.~Bump and D.~Ginzburg.
\newblock Symmetric square {$L$}-functions on {${\rm GL}(r)$}.
\newblock {\em Ann. of Math. (2)}, 136(1):137--205, 1992.

\bibitem{GG_Weiss}
W.~T. Gan and F.~Gao.
\newblock The {L}anglands-{W}eissman program for {B}rylinski-{D}eligne
  extensions.
\newblock {\em Ast\'erisque}, (398):187--275, 2018.
\newblock L-groups and the Langlands program for covering groups.

\bibitem{Kab_TensXr}
A.~C. Kable.
\newblock The tensor product of exceptional representations on the general
  linear group.
\newblock {\em Ann. Sci. \'Ecole Norm. Sup. (4)}, 34(5):741--769, 2001.

\bibitem{KudRall_WeilSieg}
S.~S. Kudla and S.~Rallis.
\newblock On the {W}eil-{S}iegel formula.
\newblock {\em J. Reine Angew. Math.}, 387:1--68, 1988.

\bibitem{Ny_LcS}
Y.~Nakata.
\newblock Local langlands correspondence for covering groups of tori, and the
  packet-indexing groups.
\newblock {\em J. Number Theory}, To appear.

\bibitem{PltRR}
V.~Platonov, A.~Rapinchuk, and I.~Rapinchuk.
\newblock {\em Algebraic groups and number theory. {V}ol. {I}}, volume 205 of
  {\em Cambridge Studies in Advanced Mathematics}.
\newblock Cambridge University Press, Cambridge, second edition, 2023.

\bibitem{Qill_hi}
D.~Quillen.
\newblock Higher algebraic {$K$}-theory. {I}.
\newblock In {\em Algebraic {$K$}-theory, {I}: {H}igher {$K$}-theories ({P}roc.
  {C}onf., {B}attelle {M}emorial {I}nst., {S}eattle, {W}ash., 1972)}, volume
  341 of {\em Lecture Notes in Math.}, pages 85--147. Springer, Berlin-New
  York, 1973.

\bibitem{Take_TwSs}
S.~Takeda.
\newblock The twisted symmetric square {$L$}-function of {$\mathrm{GL}(r)$}.
\newblock {\em Duke Math. J.}, 163(1):175--266, 2014.

\bibitem{Weil_Sieg}
A.~Weil.
\newblock Sur la formule de {S}iegel dans la th\'eorie des groupes classiques.
\newblock {\em Acta Math.}, 113:1--87, 1965.

\bibitem{Weiss_loc_tori}
M.~H. Weissman.
\newblock Metaplectic tori over local fields.
\newblock {\em Pacific J. Math.}, 241(1):169--200, 2009.

\bibitem{Weiss_LG_tori}
M.~H. Weissman.
\newblock Covers of tori over local and global fields.
\newblock {\em Amer. J. Math.}, 138(6):1533--1573, 2016.

\bibitem{Weiss_L}
M.~H. Weissman.
\newblock L-groups and parameters for covering groups.
\newblock {\em Ast\'erisque}, (398):33--186, 2018.
\newblock L-groups and the Langlands program for covering groups.

\end{thebibliography}
\end{document}